\documentclass[11pt]{amsart}
\usepackage{amsmath, amssymb, graphicx, geometry, xr}
\graphicspath{{.}{figs/}}

\usepackage[percent]{overpic}
\usepackage{varioref}
\usepackage{placeins}
\usepackage[matrix,arrow]{xy}
\usepackage{array}
\usepackage{longtable}
\usepackage{dcolumn}
\usepackage{booktabs}
\usepackage{fullpage}
\usepackage{epic,eepic}
\usepackage{fp}
\usepackage{units}
\usepackage{url}
\usepackage{hyperref}
\usepackage{color}

\usepackage{times}

\theoremstyle{plain}
\newtheorem{theorem}{Theorem}[section]

\theoremstyle{definition}
\newtheorem{definition}[theorem]{Definition}

\theoremstyle{definition}
\newtheorem{example}[theorem]{Example}

\theoremstyle{plain}
\newtheorem{proposition}[theorem]{Proposition}

\theoremstyle{plain}
\newtheorem{lemma}[theorem]{Lemma}

\theoremstyle{plain}

\theoremstyle{plain}
\newtheorem{conjecture}[theorem]{Conjecture}

\theoremstyle{plain}

\theoremstyle{plain}

\newcommand\cross{\times}
\newcommand{\MCG}{\operatorname{MCG}}
\newcommand{\Aut}{\operatorname{Aut}}
\newcommand{\Sym}{\operatorname{Sym}}
\newcommand{\Z}{\mathbf{Z}}
\newcommand{\id}{e}  
\newcommand{\s}[2]{\Sigma_{#1,#2}}
\newcommand{\co}{\operatorname{:}}

\newcommand{\PI}{\operatorname{PI}}

\everymath{\displaystyle}

\begin{document}
\thispagestyle{empty}
\title{The 27 Possible Intrinsic Symmetry Groups of Two-Component Links}

\author{Jason Cantarella}

\address{Department of Mathematics, University of Georgia,
Athens, GA 30602}
\email{cantarel@math.uga.edu, mmastin@math.uga.edu}

\author{James Cornish}
\address{Department of Mathematics, Wake Forest University, Winston-Salem, NC 27109}
 \email{cornjs0@wfu.edu, parslerj@wfu.edu}
 
\author{Matt Mastin}

\author{Jason Parsley}

\date{Revised: \today}

\begin{abstract}
We consider the ``intrinsic'' symmetry group of a two-component link $L$, defined to be the image $\Sigma(L)$ of the natural homomorphism from the standard symmetry group $\MCG(S^3,L)$ to the product $\MCG(S^3) \cross \MCG(L)$.  This group, first defined by Whitten in 1969, records directly whether $L$ is isotopic to a link $L'$ obtained from $L$ by permuting components or reversing orientations; it is a subgroup of $\Gamma_2$, the group of all such operations.

For two-component links, we catalog the 27 possible intrinsic symmetry groups, which represent the subgroups of $\Gamma_2$ up to conjugacy.  We are able to provide prime, nonsplit examples for 21 of these groups; some are classically known, some are new.  We catalog the frequency at which each group appears among all 77,036 of the hyperbolic two-component links of 14 or fewer crossings in Thistlethwaite's table. We also provide some new information about symmetry groups of the 293 non-hyperbolic two-component links of 14 or fewer crossings in the table.
\end{abstract}

\maketitle

\section{Introduction}

\subsection{Symmetries of Links}
The symmetry group of a link $L$ is a well-studied construction in knot theory.  It is defined to be the mapping class group of the pair $(S^3,L)$ and is frequently denoted $Sym(L)$.  A variety of approaches exist for computing symmetry groups for prime knots and links; see for instance \cite{MR891592, bs, MR1164115, MR1177431}.

This paper considers a different but related group of link symmetries.  The image of the group homomorphism
\begin{equation*}
\pi \co \Sym(L) = \MCG(S^3,L) \rightarrow \MCG(S^3) \cross \MCG(L).
\end{equation*}
represents what we call \emph{intrinsic} link symmetries, where we focus upon the action on $L$ itself and only record the orientation of the ambient $S^3$.  

Without using the language of mapping class groups Whitten \cite{MR0242146}, following ideas of Fox, defined the group structure of $\MCG(S^3) \cross \MCG(L)$ in 1969. They denoted this group as either $\Gamma(L)$ or $\Gamma_\mu$, where $\mu$ is the number of components of $L$.  This group (cf.\ Definition~\ref{def:whittengroup}) can be described as a semidirect product of $\Z_2$ groups recording the orientation of each component of the link $L$ with the permutation group $S_\mu$ representing component exchanges, crossed with another $\Z_2$ recording the orientation of $S^3$:
\begin{equation*}
\Gamma_\mu = \Gamma(L) =  \Z_2 \cross (\Z_2^{\mu}\rtimes S_\mu).
\end{equation*}

An element $\gamma = (\epsilon_0, \epsilon_1, \dots ,\epsilon_\mu, p)$ in $\Gamma_\mu$ acts on $L$ to produce a new link $L^\gamma$.  If $\epsilon_0 = +1$, then $L^\gamma$ and $L$ are the same as sets; however the components of $L$ have been renumbered and reoriented.  If $\epsilon_0 = -1$, the new link $L^\gamma$ is the mirror image of $L$, again with renumbering and reorientation.   We can then define the intrinsic symmetry group $\Sigma(L)$ by
\begin{equation*}
\gamma \in \Sigma(L) \iff \text{ there is an isotopy from } L \text { to } L^\gamma \text{ preserving component numbering and orientation.}
\end{equation*}
We refer to our paper \cite{intrinsicsym} for a thorough description of the history, construction, and applications of intrinsic symmetry groups.  In that paper, we find the intrinsic symmetry groups for all 48 prime links of eight or fewer crossings.

The goal of this paper is to describe how frequently the various possible intrinsic symmetry groups of two-component links occur in examples. There are 27 possible symmetry groups, representing the different subgroups of $\Gamma_2$ up to conjugacy. The subgroups are only considered up to conjugacy as conjugation corresponds to merely  relabeling and reorienting the components of the link; see section~\ref{sec:subgroups}.   Figure~\ref{fig:subgroup_lattice} displays the subgroup lattice of $\Gamma_2$. In Table~\ref{t:examples}, we present prime, nonsplit examples for 21 of these groups; some examples are already known, some are new.  Many new examples were found using the software {\tt SnapPea}; see section~\ref{sec:snappea} for details. We present in Table~\ref{tab:conj} a report on how frequently each symmetry group occurs among the 77,036 two-component hyperbolic links with 14 or fewer crossings.  

In section~\ref{sec:alt}, we restrict to the case of alternating, nonsplit two-component links, for which only 12 of the 27 symmetry groups are possible.  In particular, no alternating link with an even number of components may have full symmetry.  We realize alternating, prime examples for 11 of these 12 groups.

\section{The Whitten Group $\Gamma_2$}

We begin by giving the details of our construction of the Whitten group $\Gamma_\mu$ and the symmetry group~$\Sigma(L)$. Consider operations on an oriented, labeled link $L$ with $\mu$ components. We may reverse the orientation of any of the components of $L$ or permute the components of $L$ by any element of the permutation group~$S_\mu$. However, these operations must interact with each as well: if we reverse component 3 and exchange components 3 and 5, we must decide whether the orientation is reversed before or after the permutation. Further, we can reverse the orientation on the ambient $S^3$ as well, a process which is clearly unaffected by the permutation. To formalize our choices, we follow~\cite{MR0242146} to introduce the Whitten group of a $\mu$-component link.
\begin{definition}
Consider the homomorphism given by
\begin{equation*}
\omega:S_\mu\longmapsto~\operatorname{Aut}(\mathbf{Z}_2^{\mu+1}),\hspace{20pt}p\longmapsto\omega(p)
\end{equation*}
where $\omega(p)$ is defined as
\begin{equation*}
\omega(p)(\epsilon_0,\epsilon_1,\epsilon_2...\epsilon_\mu)=\left(\epsilon_0,\epsilon_{p(1)},\epsilon_{p(2)}...\epsilon_{p(\mu)}\right).
\end{equation*}
For $\gamma=\left(\epsilon_0,\epsilon_1,...\epsilon_\mu, p \right),$ and $\gamma'=\left(\epsilon'_0,\epsilon'_1,...\epsilon'_\mu, q \right)\in \mathbf{Z}_2^{\mu+1}\rtimes_\omega S_\mu$, we define the \emph{Whitten group} $\Gamma_\mu$ as the semidirect product $\Gamma_\mu=\mathbf{Z}_2^{\mu+1}\rtimes_\omega S_\mu$
with the group operation
\begin{align*}
\gamma\ast\gamma'&=\left(\epsilon_0,\epsilon_1,\epsilon_2...\epsilon_\mu,p\right)\ast \left(\epsilon'_0,\epsilon'_1,\epsilon'_2...\epsilon'_\mu,q \right)\\
& =\left((\epsilon_0,\epsilon_1,\epsilon_2...\epsilon_\mu)\cdot\omega(p)(\epsilon'_0,\epsilon'_1,\epsilon'_2...\epsilon'_\mu),qp \right)\\
&=\left(\epsilon_0\epsilon'_0,\epsilon_1\epsilon'_{p(1)},\epsilon_2\epsilon'_{p(2)}...\epsilon_\mu\epsilon'_{p(\mu)}, qp \right) .
\end{align*}
\label{def:whittengroup}
\end{definition}

\subsection{Intrinsic symmetries of a link}
Given a link $L$ consisting of $\mu$ oriented knots in $S^3$, we may order the knots so that
\begin{equation*}
L=K_1\cup K_2\cup \cdots \cup K_\mu.
\end{equation*}
Consider the following operations on $L$:
\begin{enumerate}
\item Permuting the $K_i$.
\item Reversing the orientation of any set of $K_i$'s.
\item Reversing the orientation on $S^3$ (mirroring $L$).
\end{enumerate}
Let $\gamma$ be a combination of any of the moves (1), (2), or (3). We think of $\gamma=(\epsilon_0,\epsilon_1,...\epsilon_\mu, p)$ as an element of the set $\mathbf{Z}_2^{\mu+1}\times S_\mu$ in the following way. Let
\begin{equation*}\epsilon_0=
\begin{cases}
-1, &\text{if $\gamma$ mirrors $L$}\\
+1, &\text{if $\gamma$ does not mirror $L$}
\end{cases}\end{equation*}
and
\begin{equation*}\epsilon_i=
\begin{cases}
-1, &\text{if $\gamma$ reverses the orientation of $K_{p(i)}$}\\
+1, &\text{if $\gamma$ does not reverse the orientation of $K_{p(i)}$}
\end{cases}\end{equation*}
Lastly, let $p\in S_\mu$ be the permutation of the $K_i$ associated to $\gamma$.  To be explicitly clear, permutation $p$ permutes the labels of the components; the component originally labeled $i$ will be labeled $p(i)$ after the action of $\gamma$.  

For each element, $\gamma$ in $\mathbf{Z}_2^{\mu+1}\times S_\mu$, we define
\begin{equation}
\label{eq:Lgamma}
L^\gamma=\gamma(L)=\epsilon_1K_{p(1)}^{(*)}\cup \epsilon_2K_{p(2)}^{(*)}\cup\cdot\cdot\cdot \cup \epsilon_\mu K_{p(\mu)}^{(*)} =\bigcup_{i=1}^{\mu}\epsilon_i K_{p(i)}^{(*)}
\end{equation}
\\where $-K_i$ is $K_i$ with orientation reversed, $K_i^*$ is the mirror image of $K_i$ and the $(*)$ appears above if and only if $\epsilon_0 = -1$. Note that the $i$th component of $\gamma(L)$ is $\epsilon_i K_{p(i)}^{(*)}$ the possibly reversed or mirrored $p(i)$th component of $L$.  Since we are applying $\epsilon_i$ instead of $\epsilon_{p(i)}$ to $K_{p(i)}$ we are taking the convention of first permuting and then reversing the appropriate components.

We can now define the subgroup of $\Gamma(L)$ which corresponds to the symmetries of the link $L$.

\begin{definition}Given a link, $L$ and $\gamma\in \Gamma(L)$, we say that \emph{$L$ admits $\gamma$} when there exists an isotopy taking each component of $L$ to the corresponding component of $L^{\gamma}$ which respects the orientations of the components.  We define as the \emph{intrinsic symmetry group} (also called the \emph{Whitten symmetry group}) of $L$,
\begin{equation*}
\Sigma(L):=\{\gamma\in \Gamma(L) |~\text{$L$ admits $\gamma$}\}.
\end{equation*}
\end{definition}

The intrinsic symmetry group $\Sigma(L)$ is a subgroup of $\Gamma_\mu$, and its left cosets represent the different isotopy classes of links $L^\gamma$ among all symmetries $\gamma$.  By counting the number of cosets, we determine the number of (labeled, oriented) isotopy classes of a particular prime link.  

Next, we provide a few examples of symmetry subgroups.  Recall that the first Whitten group $\Gamma_1 = \Z_2 \times \Z_2$ has order four and that $\Gamma_2 = \Z_2 \cross (\Z_2 \cross \Z_2 \rtimes S_2)$ is a nonabelian 16 element group.

\begin{example} Let $L=3_1$, a trefoil knot. It is well known that $L\sim-L$ and $L^*\sim-L^*$, but $L\nsim L^*$, so we have $\Sigma(3_1)=\{(1,1, \id), (1,-1, e)\}$. This means that the two cosets of $\Sigma(3_1)$ are $\{(1,1,\id), (1,-1, \id)\}$ and $\{(-1,-1, \id), (-1,1,\id)\}$, and there are two isotopy classes of $3_1$ knots.  A trefoil knot is thus \emph{invertible}.
\hfill{$\Diamond$}
\end{example}

\begin{example}
Let $L=8^2_{12}$, whose components are an unknot $K_1$ and a trefoil $K_2$.  Since the components $K_1$ and $K_2$ are of different knot types, we conclude that no symmetry in $\Sigma(8^2_{12})$ can contain the permutation $(12)$. As we describe in \cite{intrinsicsym}, certain symmetries are impossible for alternating links with nonzero self-writhe, such as $8^2_{12}$.  In particular, such a link never admits a symmetry that reverses the ambient orientation, i.e., one with $\epsilon_0=-1$.  Together these two obstructions rule out 12 of the 16 elements in $\Gamma_2$, so we may conclude that $\Sigma(8^2_{12}))$ has order four or less.

It is not difficult to find isotopies that show that $L$ is purely invertible, meaning isotopic to $-L=-K_1\cup-K_2$, and individually invertible in the first component, meaning isotopic to $-K_1\cup K_2$.  Thus, both $(1,-1,-1,e)$ and $(1,-1,1,e)$ are included in the Whitten group for $L$.  Hence, $\Sigma(L)$ is the four element group $\s{4}{2}=\{(1,1,1,e), (1,-1,-1,e), (1,-1,1,e), (1,1,-1,e)\}$.  There are four cosets of this four element group in the 16 element group $\Gamma_2$, so there are four (labeled, oriented) isotopy classes of $8^2_{12}$ links.  
\hfill{$\Diamond$}
\end{example}

\subsection{Notation}
We introduce some notation and names for commonly occurring symmetries.   Let $m = (-1,1,1,e)$ describe a \emph{mirror} of the link and $\rho = (1,1,1,(12))$ denote a \emph{pure exchange} symmetry.  Inverting component $k$ is denoted by $i_k = (1,\varepsilon_1, \varepsilon_2,e)$ where $\varepsilon_k = -1$ and $\varepsilon_{3-k}=1$.  Let $a_k = \rho \ast i_k = (1,\varepsilon_1, \varepsilon_2,(12))$; observe that these elements have order 4.  A link has \emph{pure invertible} symmetry if $\PI = (1,-1,-1,e) \in \Sigma(L)$.  Finally, let $b = \PI \ast \rho = \rho \ast \PI = (1,-1,-1,(12))$.  We shall sometimes use this notation in a product form, e.g., $ma_1 = (-1,-1,1,(12))$.

Henceforth, `symmetry group' will always refer to the intrinsic symmetry group $\Sigma(L)$.  We use the notation from \cite{intrinsicsym} for symmetry groups: $\s{k}{j}$ is the $j$th subgroup of order $k$ in our subgroup lattice for $\Gamma_2$.

\subsection{Subgroups of $\Gamma_2$} \label{sec:subgroups}
In order to classify the possible symmetry groups of two-component links, we must find the subgroups of $\Gamma_2$.  For any $\gamma \in \Gamma_2$, the symmetry groups $\Sigma(L)$ and $\Sigma(L^\gamma)$ are conjugate via $\gamma$.  Therefore if a link $L$ has symmetry group $\Sigma$, which is conjugate to group $\Sigma'$, then there exists some $\gamma \in \Gamma_2$ such that $L^\gamma$ has symmetry group $\Sigma'$; link $L^\gamma$ is merely a relabeled and reoriented copy of $L$.  
 Thus, the truly different symmetry groups are represented by the conjugacy classes of $\Gamma_2$ subgroups.  

\begin{proposition}
There are 35 subgroups of $\Gamma_2$ and 27 mutually nonconjugate subgroups. A maximal set of nonconjugate subgroups consists of 7 subgroups of order two denoted $\s{2}{1}$ through $\s{2}{7}$, 11 subgroups of order four denoted $\s{4}{1}$ through $\s{4}{11}$ and seven subgroups of order eight denoted $\s{8}{1}$ through $\s{8}{7}$, as well as the trivial subgroup and the full group. Generators for these groups and the lattice structure of the subgroups appear in Figure~\ref{fig:subgroup_lattice}.
\end{proposition}

We note that this proposition corrects a mistake in the literature \cite{MR867798}, in which it is reported that 28 possible symmetry groups exist.

\begin{proof}
From the semidirect product structure we can see that all but four of the 15 non-identity elements  of $\Gamma_2$ are order two.  The exceptions are $a_1 = (1,-1,1,(12))$,  $a_2 = (1,1,-1,(12))$, $ma_1$, and $ma_2$, each of which has order four.  Thus, there are 11 subgroups of $\Gamma_2$ of order two.  Via computations in $\Gamma_2$, we proceed to find 15 subgroups of order four and 7 subgroups of order eight.  Combined with the trivial subgroup and all of $\Gamma_2$, we arrive at a total of 35 subgroups.  

However, some of these subgroups occur in conjugate pairs.  For example, the subgroups $\langle m ,\rho \rangle = \{ e , m , \rho , m\rho\}$ and $\langle  m , b \rangle = \{ e , m , b , m b\}$ are conjugate four-element subgroups, since $i_1(m)i_1^{-1} = m$, $i_1(\rho)i_1^{-1} = b$, and $i_1(m\rho)i_1^{-1} = mb$.  

There are four conjugate pairs of four-element subgroups and four conjugate pairs of eight-element subgroups.  Therefore precisely 27 mutually nonconjugate subgroups of $\Gamma_2$ exist.  Our calculations (by hand) of these $\Gamma_2$ subgroups and the corresponding lattice structure given in Figure~\ref{fig:subgroup_lattice} have been verified using the software {\tt Magma}.  
\end{proof}

\footnotesize
\begin{center}
\begin{figure}[t]
     \def\svgwidth{\columnwidth}
     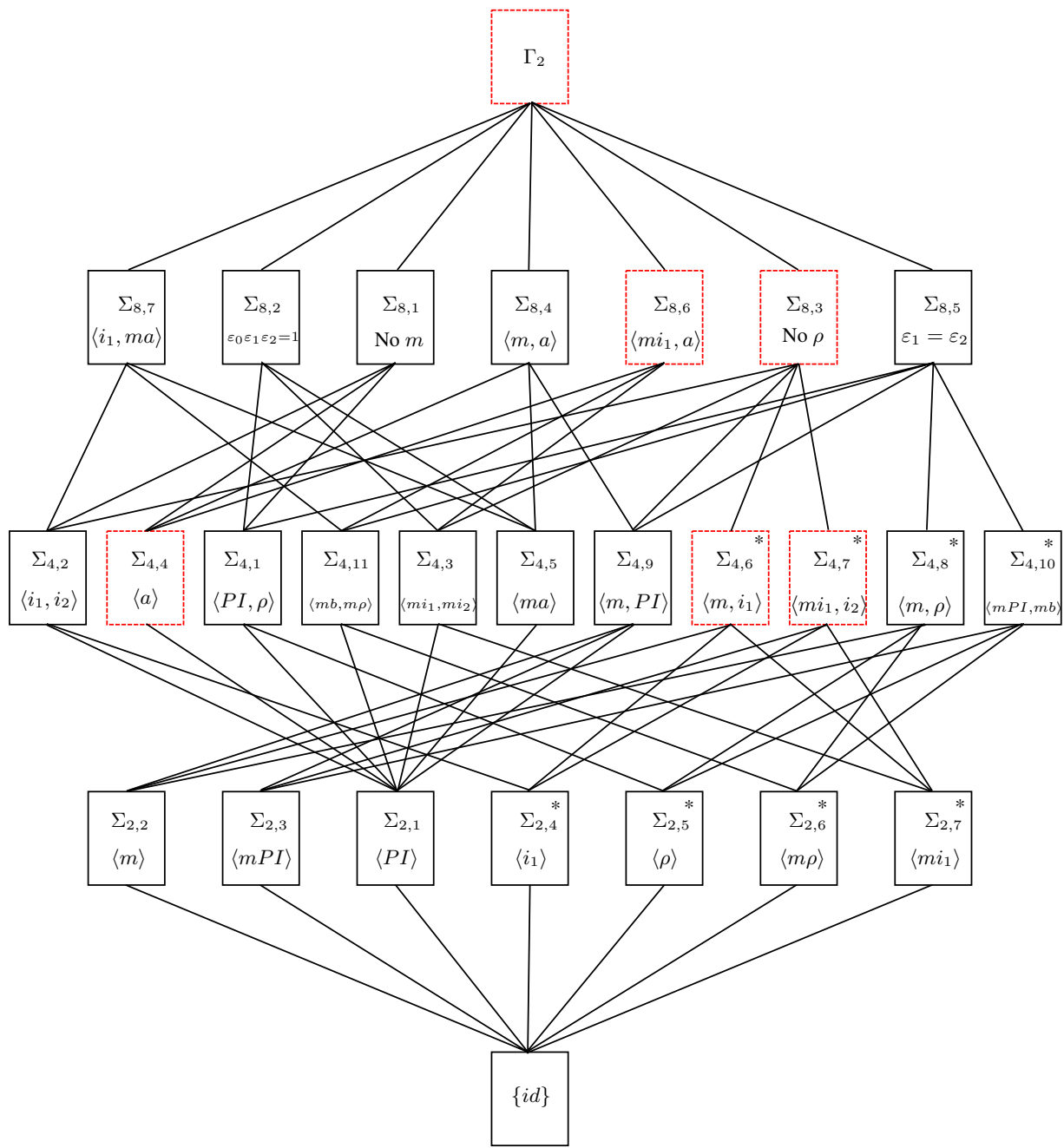
     \caption{The 27 subgroups (up to conjugacy) of the Whitten group $\Gamma_2$ are depicted above as a subgroup lattice.  There are 8 pairs of conjugate subgroups of $\Gamma_2$; above, one representative is displayed for each conjugate pair; these are denoted by $\ast$.  For 21 of these subgroups, we provide an example realizing this symmetry group.  The 6 subgroups for which no example has been found are depicted with a red dashed border.}
\label{fig:subgroup_lattice}
\end{figure}
\end{center}
\normalsize

\section{Computational Examples} \label{sec:snappea}

For a hyperbolic link $L$, we may compute its symmetry group $\Sigma(L)$ using the software \texttt{SnapPea}.  This software can calculate the mapping class group $\MCG(S^3 \setminus L)$ of the link complement. The elements of $\MCG(S^3,L)$ which extend through the boundary tori to automorphisms of all of $S^3$ form a copy of $\MCG(S^3,L)$ inside $\MCG(S^3 \setminus L)$. We can detect such maps using the following standard lemma.

\begin{lemma}
A map in $\Aut(S^3 \setminus L)$ extends to all of $S^3$ if and only if it sends meridians of the boundary tori to meridians. Moreover, any two such extensions are isotopic.
\end{lemma}

We utilize the \texttt{Python} front end \texttt{SnapPy} for \texttt{SnapPea} written by Marc Culler and Nathan Dunfield. To each map on the boundary tori of the link complement, \texttt{SnapPy} assigns a collection of $\mu$ matrices along with a permutation element which records how the components of the link were permuted.  Each matrix is $2 \times 2$ and records the images of the meridians and longitudes of the appropriate component, along with the orientation of the ambient space. The effect of the map on the orientation $\epsilon_i$ of the given component and the orientation $\epsilon_0$ of $S^3$  is given by the rules below.  Note that if the matrix for one  boundary torus indicates that the orientation on $S^3$ is reversed, then so will the matrices for all other boundary tori, since these matrices result from restricting a single map on $S^3$.

\begin{equation*}
\begin{array}{ccccccc}
\left(
\begin{array}{cc}
 1 & n \\ 0 & 1
\end{array}
\right)
 &&
\begin{pmatrix}
 -1 & n \\ 0 & 1
\end{pmatrix}
&&
\begin{pmatrix}
 1 & n \\ 0 & -1
\end{pmatrix}
&&
\begin{pmatrix}
 -1 & n \\ 0 & -1
\end{pmatrix}
\\
\Downarrow &&
\Downarrow &&
\Downarrow &&
\Downarrow 
\\
 \epsilon_0 = 1, \epsilon_i = 1  & \qquad &
 \epsilon_0 = -1,  \epsilon_i = 1  & \qquad & 
\epsilon_0 = -1,  \epsilon_i = -1  & \qquad & 
 \epsilon_0 = 1,  \epsilon_i = -1   
\end{array}
\end{equation*}

Using these rules, it is easy to extract $\Sigma(L)$ from \texttt{SnapPy}. We computed the symmetry group for all 77,036 two-component hyperbolic links of 14 and fewer crossings in \texttt{SnapPy}'s database. Table~\ref{tab:conj} shows the census of symmetry groups of hyperbolic links found using \texttt{SnapPy}. The data file containing the Whitten group elements for these links is included in the \texttt{Arxiv} data repository and it is a future project to incorporate our computational techniques into \texttt{SnapPy}.

\begin{table}[ht]
\begin{center}
\begin{tabular}{ccc}
\begin{tabular}[t]{cc}
\toprule
Group & Number of Links \\
\midrule
$\{id\} $ & 53484 \\
& \\
& \\
$\Sigma_{2,1}$    &      17951 \\
$\Sigma_{2,2}$    &      7 \\
$\Sigma_{2,3}$    &      9 \\
$\Sigma_{2,4}$  &1336\\
$\Sigma_{2,5} $ &418\\
$\Sigma_{2,6} $& 3\\
$\Sigma_{2,7} $ & 123\\
\end{tabular}

&
\begin{tabular}[t]{cc}
\toprule
Group & Number of Links \\
\midrule
$\Sigma_{4,1} $&  1396\\
$\Sigma_{4,2} $ & 2167\\
$\Sigma_{4,3} $ & 24\\
$\Sigma_{4,4} $&  0\\
$\Sigma_{4,5} $ & 12\\
$\Sigma_{4,6} $&  0\\
$\Sigma_{4,7} $ & 0\\
$\Sigma_{4,8} $&  2\\
$\Sigma_{4,9} $ & 11\\
$\Sigma_{4,10} $ & 1\\
$\Sigma_{4,11} $  &4\\
\end{tabular}
&
\begin{tabular}[t]{cc}
\toprule
Group & Number of Links \\
\midrule
$\Sigma_{8,1} $&52\\
$\Sigma_{8,2} $ & 25\\
$\Sigma_{8,3} $ & 0\\
$\Sigma_{8,4} $&  1\\
$\Sigma_{8,5} $  &2\\
$\Sigma_{8,6} $&  0\\
$\Sigma_{8,7} $&  8\\
& \\
& \\
$\Gamma_2$ & 0 \\
\end{tabular} \\

\end{tabular}
\medskip
\caption[Census of symmetry groups for hyperbolic two-component links.] {The number of links admitting each symmetry group among the 77,036 hyperbolic two-component links included in the Thistlethwaite table distributed with \texttt{SnapPy}. Among these links, almost two-thirds had no symmetry (trivial symmetry group) and there were no examples with full symmetry.}
\label{tab:conj}
\end{center}
\end{table}

There are 293 non-hyperbolic two-component links with up to 14 crossings in Thistlethwaite's link table. For these links, we were unable to compute the symmetry groups using SnapPea. We were able to get partial information about the symmetry groups using the {\tt Mathematica} package {\tt KnotTheory}~\cite{knottheory}. For each link $L$ and each Whitten group element $\gamma$, we computed a number of invariants of $L$ and $L^\gamma$ in an attempt to rule out $\gamma$ as an element of $\Sigma(L)$. To rule out ``exchange'' symmetries, we applied two tests. First, we computed the Jones polynomials of each component of $L$ to try to rule out exchanges between components of different knot types. If that failed, we turned to the ``satellite lemma'', which we use often in \cite{intrinsicsym}.

\begin{lemma}
\label{lem:satellitelemma}
Suppose that $L(K,i)$ is a satellite of $L$ constructed by replacing component $i$ with a knot or link $K$. Then $L$ cannot have a pure exchange symmetry exchanging components $i$ and $j$ unless $L(K,i)$ and $L(K,j)$ are isotopic.
\end{lemma}

\begin{proof}
Such a pure exchange would carry an oriented solid tube around $L_i$ to a corresponding oriented solid tube around $L_j$. If we imagine $K$ embedded in this tube, this generates an isotopy between~$L(K,i)$ and~$L(K,j)$.
\end{proof}

The point of this lemma is that we can often distinguish $L(K,i)$ and $L(K,j)$ using classical invariants which are insensitive to the original labeling of the link. In our computations, we replaced component $i$ of $L$ and $L^\gamma$ with a $(2,2i)$ torus link (for $i > 0$) and compared the Jones polynomials of the resulting links. An example is shown in Figure~\ref{fig:satellite}.

\begin{figure}
\hfill
\includegraphics[height=1.5in]{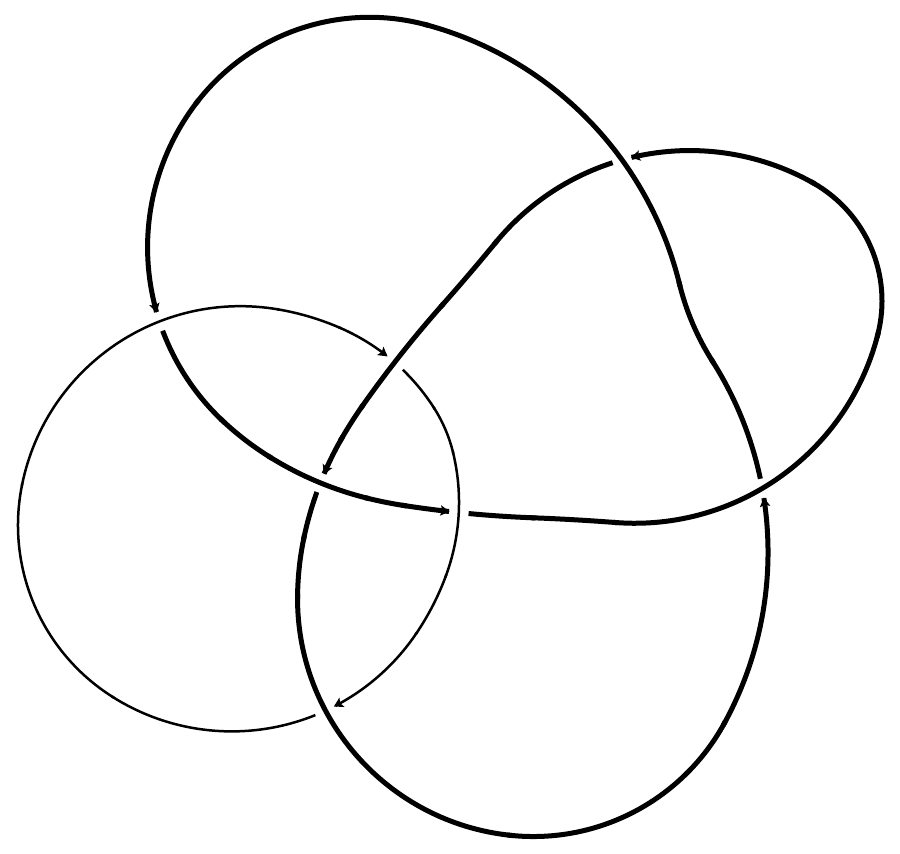}
\hfill
\includegraphics[height=1.5in]{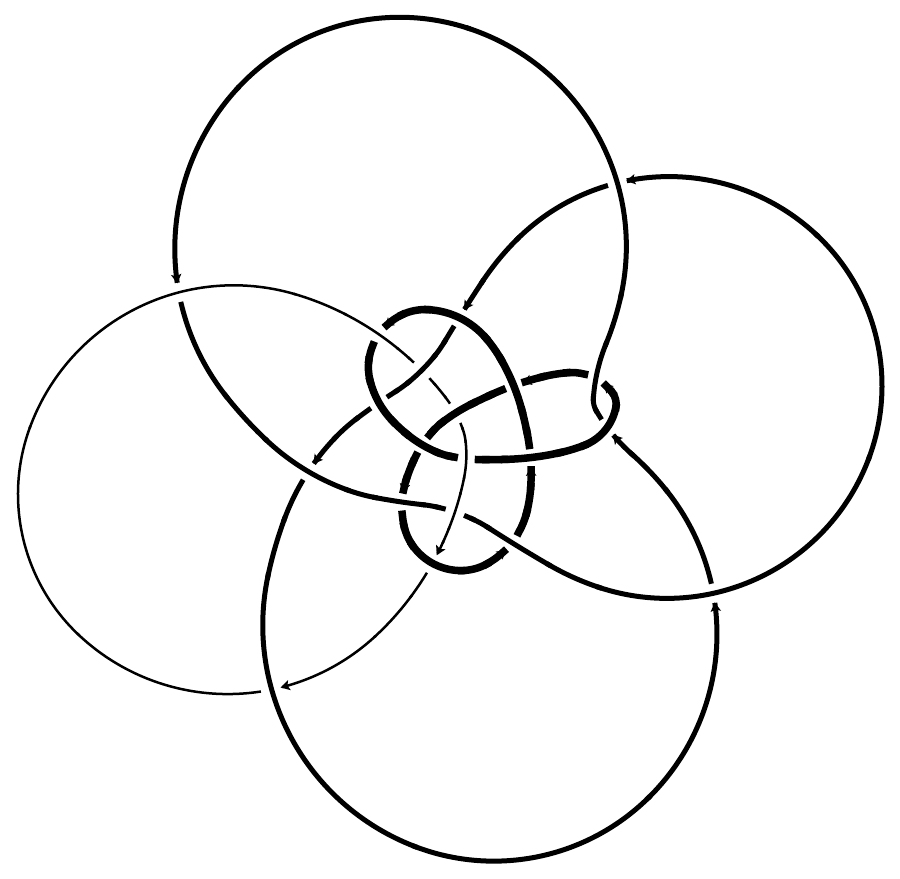}
\hfill
\includegraphics[height=1.5in]{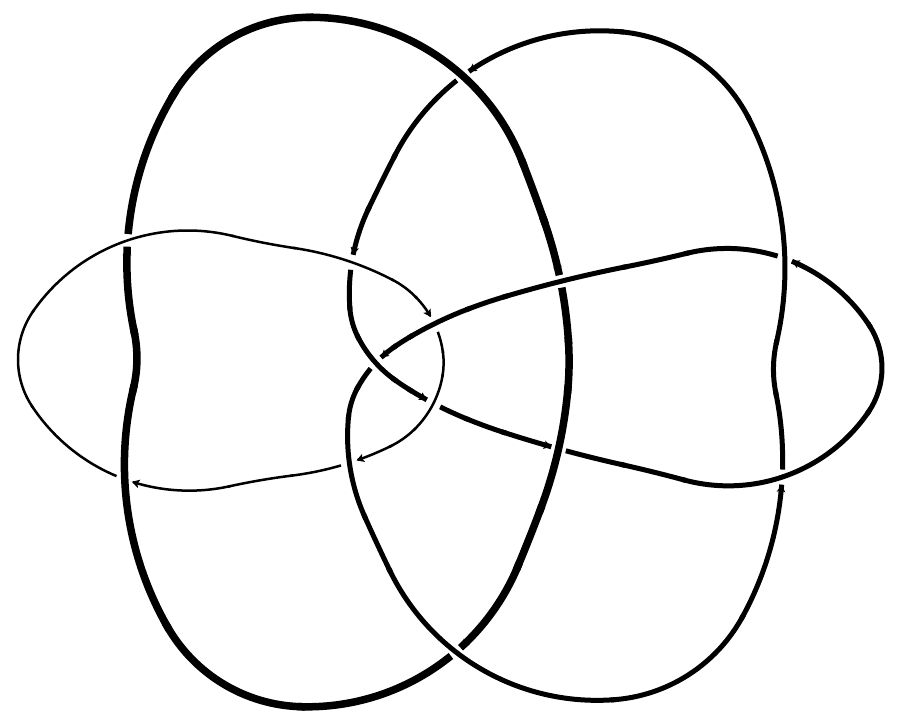}
\hfill
\hphantom{.}
\caption{The link $7^2_6$ (or $7a1$ in Thistlethwaite's notation) has two unknotted components with linking number zero. To decide whether this link admits a pure exchange symmetry, we construct two satellite links. The center link has one component replaced by a Hopf link while the other is unframed, while the right link has the other component replaced by a Hopf link. If the original $7^2_6$ admits a pure exchange symmetry, these satellites would be isotopic. However the Jones polynomial of the center link is $a^{10}-2 a^9+a^8-a^6+2 a^5-a^4-\frac{1}{a^4}+2 a^3+\frac{1}{a^3}+\frac{1}{a^2}-\frac{1}{a}+2$ while the Jones polynomial of the right-hand link is $a^7-2 a^6+2 a^5-2 a^4+2 a^3-\frac{1}{a^3}+a^2+\frac{2}{a^2}-\frac{2}{a}+3$. This proves that $7^2_6$ does not admit a pure exchange symmetry.
\label{fig:satellite}}
\end{figure}

Ruling out elements of the Whitten group using these methods gave us a subgroup of $\Gamma_2$ guaranteed to contain $\Sigma(L)$. We found $9$ links with $\Sigma(L) < \Sigma_{4,1}$, $89$ links with $\Sigma(L) < \Sigma_{4,2}$, $194$ links with $\Sigma(L) < \Sigma_{2,1}$ and 1 link with $\Sigma(L) < \Sigma_{8,1}$. Consulting the subgroup lattice of $\Gamma_2$ along with Table~\ref{tab:conj}, we saw that only this last link, the 194th 14-crossing two component non-hyperbolic link in Thistlethwaite's table, might have had a new symmetry group which was not already represented by a hyperbolic link. The potential group was $\Sigma_{4,4} < \Sigma_{8,1}$. However, rearranging the diagram as in Figure~\ref{fig:danger} ruled out $\Sigma_{4,4}$ by revealing a pure exchange symmetry.

\begin{figure}
\hfill
\includegraphics[height=2.5in]{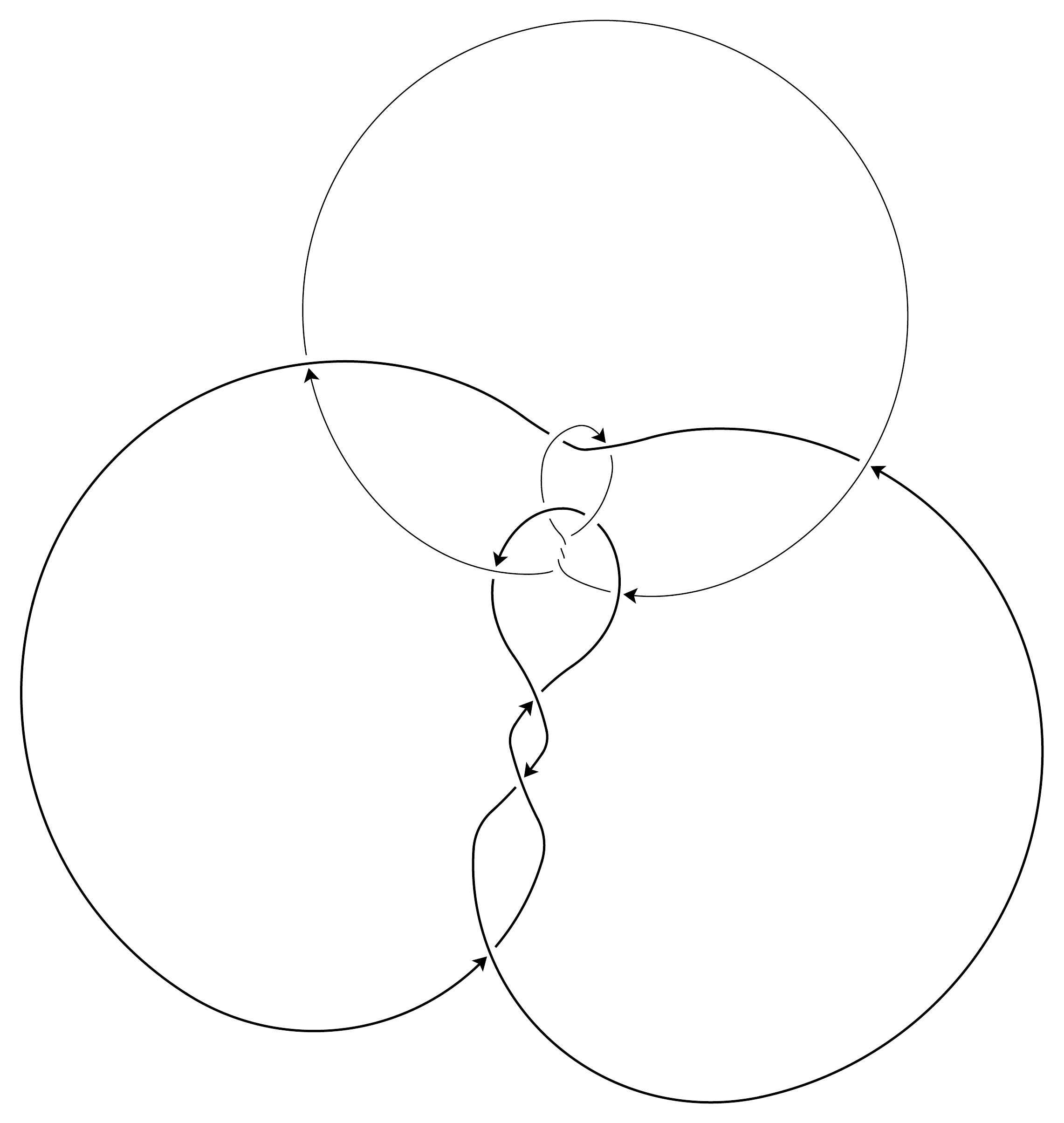}
\hfill
\includegraphics[height=2.5in]{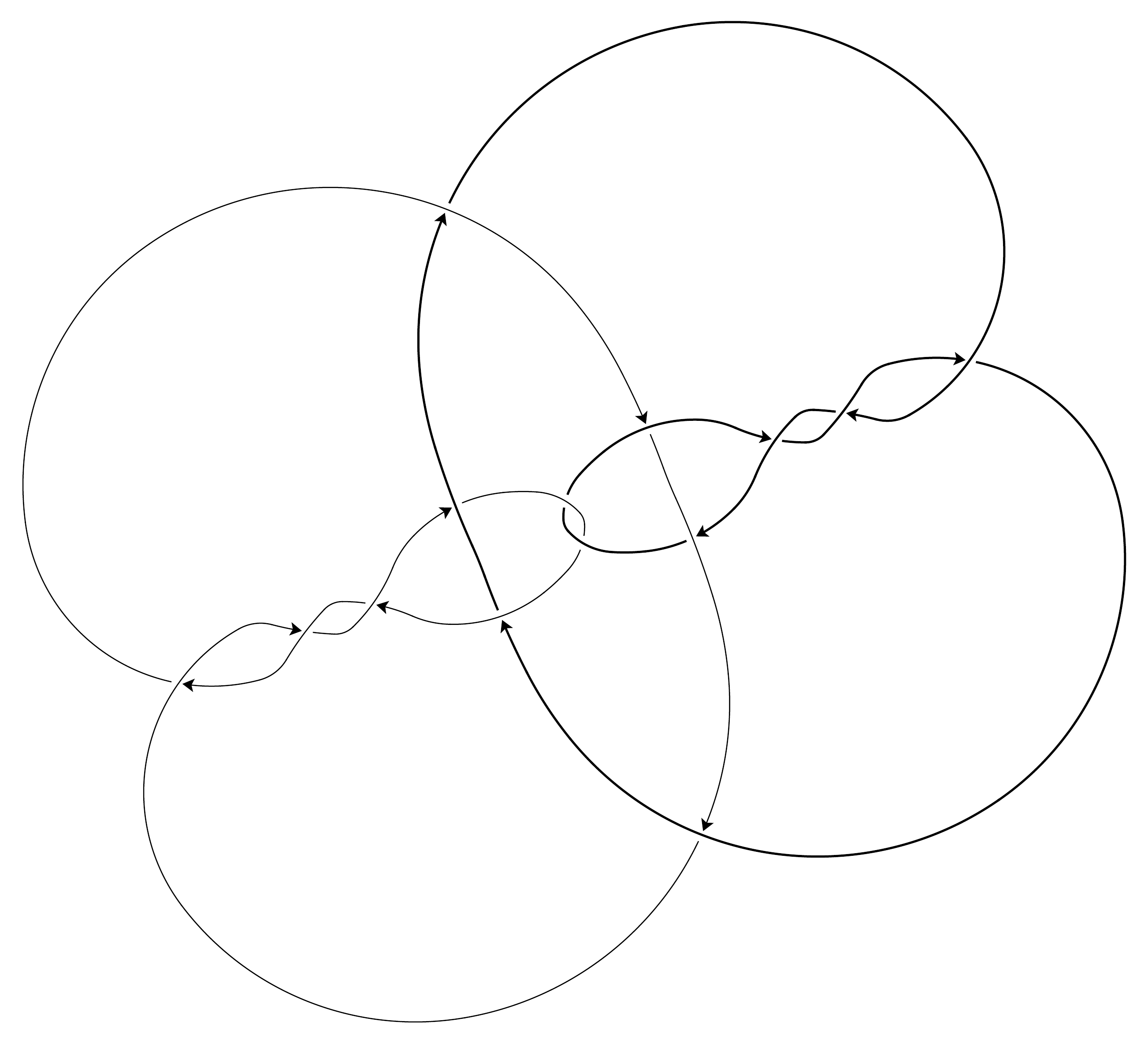}
\hfill
\hphantom{.}
\caption{The 194th non-hyperbolic two-component link with 14 crossings in Thistlethwaite's table has $\Sigma(L) < \Sigma_{8,1}$ according to our polynomial and satellite lemma calculations. The original diagram of this link is shown at left. Since $\Sigma_{4,4} < \Sigma_{8,1}$, this link could have the ``missing'' symmetry group $\Sigma_{4,4}$. However, rearranging the diagram as shown on the right reveals that this link has pure exchange symmetry. Since pure exchange is not part of $\Sigma_{4,4}$, we see that this link cannot be an example of a link with $\Sigma_{4,4}$ symmetry.
\label{fig:danger}}
\end{figure}

Having obtained supergroups of the $\Sigma(L)$ groups for these non-hyperbolic links, we turned to obtaining subgroups of $\Sigma(L)$. To do so, we viewed each diagram as a polyhedral decomposition of the 2-sphere and enumerated the combinatorial symmetries of each polyhedron using \texttt{Mathematica}. Those symmetries which extended to symmetries of the link after crossing information was taken into account were classified according to the Whitten element they represented. We called these \emph{diagrammatic} symmetries of each link. Of course, different diagrams are expected to reveal additional symmetries, so the diagrammatic symmetry group of a diagram of a link is only a subgroup of $\Sigma(L)$. In most cases, we still have a number of potential symmetries which may or may not be present for the link. However, we found 70 cases where the diagrammatic symmetry groups represented all of $\Sigma(L)$ since they agreed with the ``supergroups'' computed earlier: 6 links with $\Sigma(L) = \Sigma_{4,1}$, 4 links with $\Sigma(L) = \Sigma_{4,2}$, and 60 links with $\Sigma(L) = \Sigma_{2,1}$.

\section{Examples of links with particular symmetry groups}

Table~\ref{t:examples} lists examples of prime, nonsplit links by their intrinsic symmetry groups.  We present examples for 21 of the 27 different symmetry groups.

\begin{table}[h]
\begin{tabular}{ll}
\toprule
Symmetry Group & Example \\
\midrule
$\{id\}$&11a164 \\ 
\midrule
$\s{2}{2}=\langle m \rangle$&\small{DT[8,10,-16:12,14,24,2,-20,-22,-4,-6,-18]}\\

$\s{2}{3}=\langle mPI \rangle$& \small{DT[8,10,-14:12,22,24,-18,-20,-4,-6,-16,2]}\\
 
$\s{2}{1}=\langle PI \rangle$&7a2\\
 
$\s{2}{4}=\langle i_1 \rangle$&10a98\\
 
$\s{2}{5}=\langle \rho \rangle$& \small{DT[14,16,44,20,4,2,18,10,12,8,28:36,38,22,42,26,24,40,32,34,30,6]}\\
 
$\s{2}{6}=\langle m\rho \rangle$& \small{DT[44,-20,-16,-22,-14,-4,-6,-10,-12,-28,-8:2,34,-18,24,42,26,40,30,32,36,38]} \\
 
$\s{2}{7}=\langle mi_1 \rangle$&10a81\\
\midrule
$\s{4}{2}$&7a3\\
 
$\s{4}{4}$& {\it No example known}\\
 
$\s{4}{1}$&4a1\\
 
$\s{4}{11}$& \small{DT[16,-6,-12,-20,-22,-4,28:2,14,26,-10,-8,18,24]}\\
 
$\s{4}{3}$& 10n46, 10a56 \\ 

$\s{4}{5}$& \small{DT[14,6,10,16,4,20:22,8,2,24,12,18]}\\
 
$\s{4}{9}$&10n36\\
 
$\s{4}{6}$& {\it No example known}\\
 
$\s{4}{7}$& {\it No example known}\\
 
$\s{4}{8}$& \small{DT[10,-14,-18,24:2,28,-4,-12,-20,-6,-16,26,8,22]}\\
 
$\s{4}{10}$& \small{DT[10,-14,-20,24:2,28,-16,-4,-12,-6,-18,26,8,22]}\\
\midrule
$\s{8}{7}$&10n59\\
 
$\s{8}{2}$&2a1\\
 
$\s{8}{1}$&5a1\\
 
$\s{8}{4}$& \small{DT[14,-16,22,-28,24,-18:12,-20,-10,-2,26,8,5,-6]}\\
 
$\s{8}{6}$&{\it No example known}\\
 
$\s{8}{3}$&{\it No example known}\\
 
$\s{8}{5}$& \small{DT[10,-14,-18,22:2,24,-4,-12,-6,-16,8,10]}\\
\midrule
$\Gamma_2$&{\it No example known}\\
\bottomrule
\end{tabular}
\smallskip
\caption{Examples of prime, nontrivial two-component links are currently known for 21 of the possible 27 intrinsic symmetry groups.}
\label{t:examples}
\end{table}

The example links for the groups $\{id\}$ and $\s{8}{5}$  were originally found by Hillman \cite{MR867798}.  He listed examples for many of the 12 subgroups that do not include $i_1$ or $i_2$ in that paper as well.  The examples for the groups $\s{8}{1}$, $\s{8}{2}$, $\s{4}{1}$, $\s{4}{2}$, and $\s{2}{1}$ were generally known; for example, proofs that these links exhibit the given symmetry groups are found in \cite{intrinsicsym}.  

The examples for $\langle \rho \rangle$, $\langle m\rho \rangle$ and $\s{8}{7}$ are found in Figure~\ref{fig:new_examples} and are new to the literature; we prove that these symmetry groups are correct below.  All other examples have been found using {\tt SnapPea} and are new to the literature.

\begin{figure}
\includegraphics[height=2.0in]{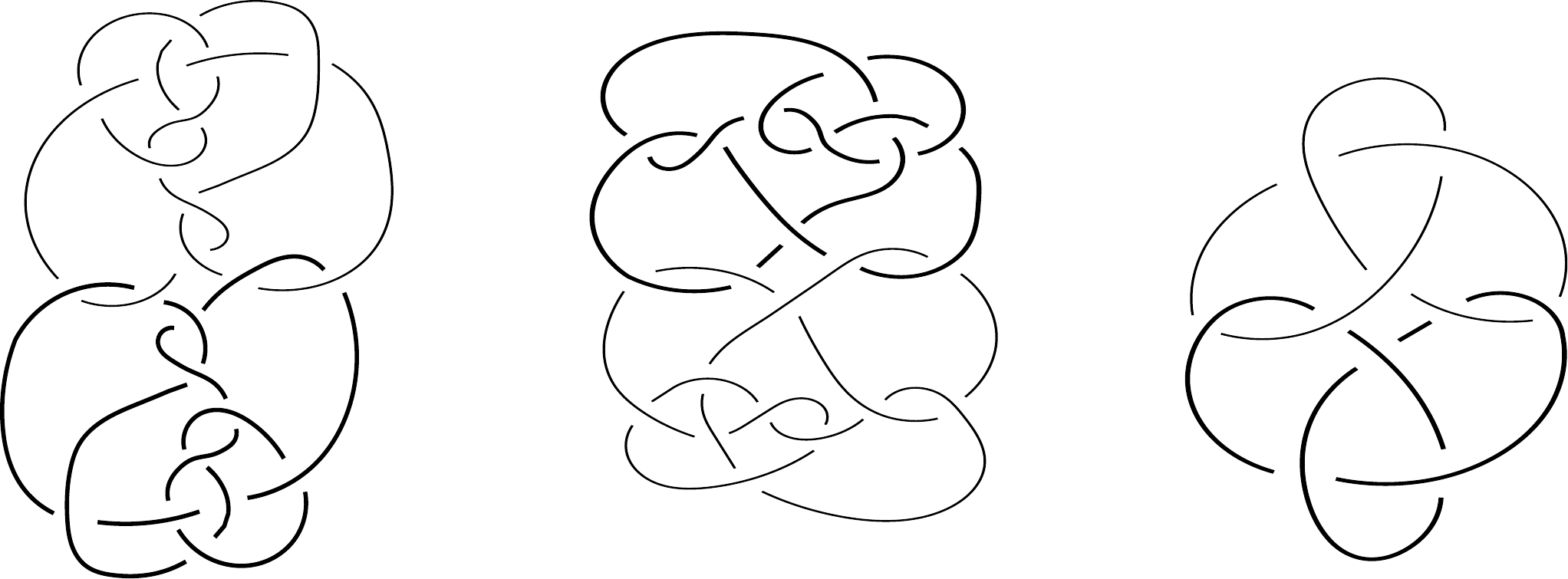}
\caption{Three new example links for symmetry groups, not found using \texttt{SnapPea}.  From left to right, these links exhibit the symmetry groups $\langle \rho \rangle$, $\langle m\rho \rangle$ and $\s{8}{7}$.
\label{fig:new_examples}}
\end{figure}

\begin{theorem}
The three links in Figure~\ref{fig:new_examples} possess the symmetry groups $\langle \rho \rangle$, $\langle m\rho \rangle$ and $\s{8}{7}$, respectively.
\end{theorem}

\begin{proof}
The link on the far left in Figure~\ref{fig:new_examples}  is comprised of two components of $9_{32}$, while the link in the center is comprised of knots $9_{32}$ and $9_{32}^*$.  The knot $9_{32}$ has no symmetry, therefore the only possible nontrivial symmetries for these links are $\rho$ and $m\rho$ respectively.  These two links achieve said symmetries through a $180$ degree rotation about a line perpendicular to the page.  

The two components of the link 10n59 in the far right in Figure~\ref{fig:new_examples} are a left- and a right-handed trefoil ($3_1$ and $3_1^*$).  Since the trefoil is not mirrorable, we cannot mirror link 10n59 without permuting components, and vice-versa.    Therefore its symmetry group is a subgroup of $\s{8}{7} = \langle i_2, m\rho \rangle$.  The link exhibits the $m\rho$ symmetry through a  $180$ degree rotation about a line perpendicular to the page; Figure~\ref{fig:10n59} shows it also admits the symmetry $i_2$.
Thus, $\Sigma(10n59) = \s{8}{7}$. \qedhere
\begin{figure}[ht]
\includegraphics[height=3.0in]{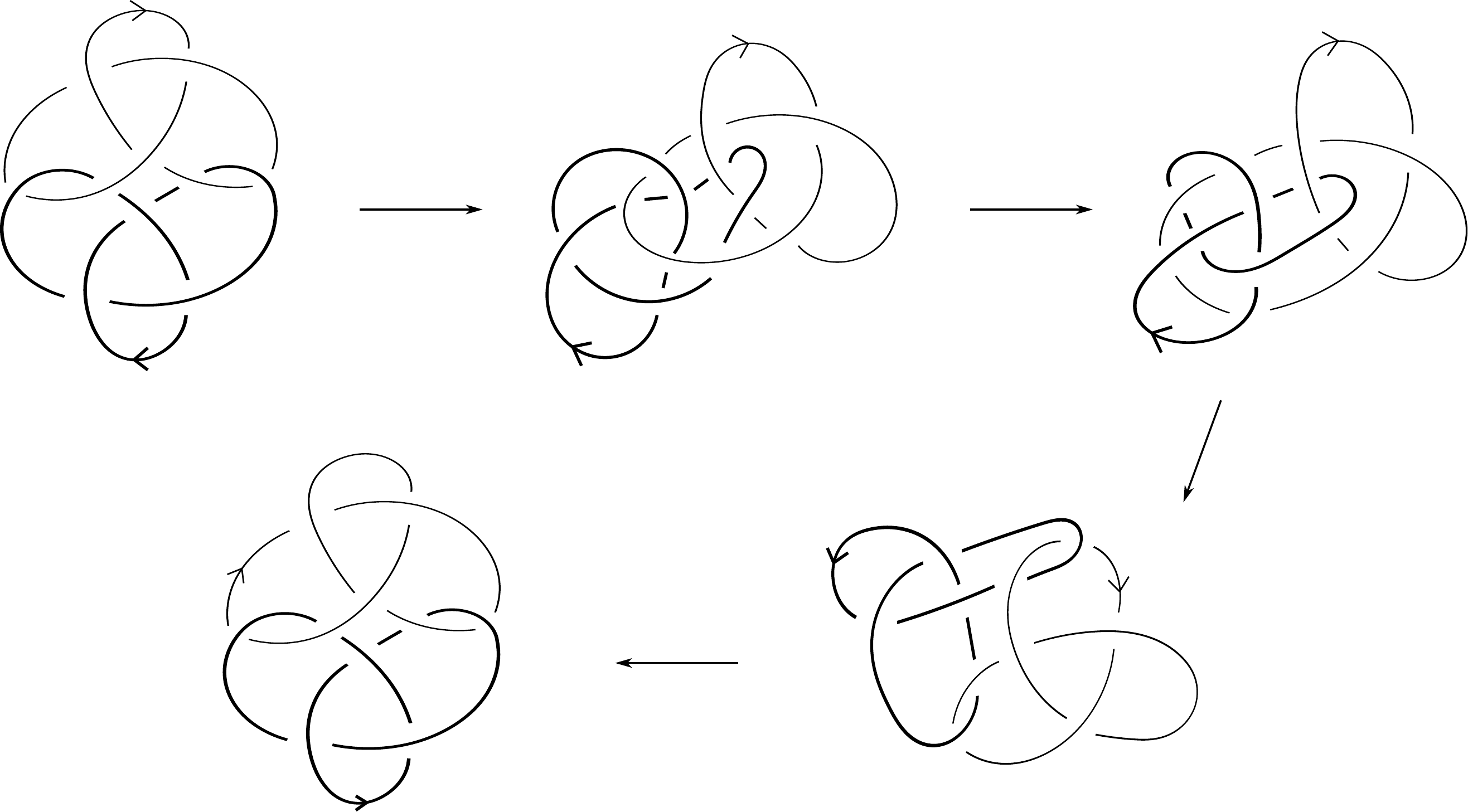}
\caption{An isotopy diagram showing that link 10n59 admits symmetry $i_2$ -- we may reverse the orientation of the second component.
\label{fig:10n59}}
\end{figure}

\end{proof}

\section{Intrinsic symmetry groups of alternating links} \label{sec:alt}

Alternating, nonsplit two-component links only have 12 possible intrinsic symmetry groups; 15 groups will never occur when the link is alternating.  To understand why, we will utilize a property of Conway polynomials \cite{MR0258014}.  Though this result is classically known, we provide a proof here for completeness.  

\begin{lemma}
Let $L$ be a link with $c$ components.  Consider a symmetry $\gamma=(-1,\varepsilon,\varepsilon,...,\varepsilon ,p)$ that reverses the orientation of either all or no components of $L$.  Then, the Conway polynomial of $L$ changes sign according to the formula  $\nabla(L^\gamma)=(-1)^{c+1}\nabla(L)$.
\end{lemma}

\begin{proof}
We shall proceed by induction on the number of components.  For knots ($c=1$), this is well known.  Let $\gamma$ be defined as above.  Let $K$ be one of the components of $L$.  Resolve each of the $n$ overcrossings that $K$ has with the other components of $L$ according to the skein relation of the Conway polynomial; similarly, resolve the corresponding $n$ undercrossings in $L^\gamma$.  As this will turn every crossing $K$ has with $L$ or $L^\gamma$ into over or under crossings, we eventually get down to split links $L-K$ and $L^\gamma-K$.  We obtain
\begin{equation*}
\nabla(L)=(-1)^{i_1}z\nabla(L_{0,1})+(-1)^{i_2}z\nabla(L_{0,2})+...+(-1)^{i_n}z\nabla(L_{0,n})+\nabla(L-K)
\end{equation*}
and
\begin{equation*}
\nabla(L^\gamma)=(-1)^{i_1+1}z\nabla(L_{0,1}^\gamma)+(-1)^{i_2+1}z \nabla(L_{0,2}^\gamma)+...+(-1)^{i_n+1}z\nabla(L_{0,n}^\gamma)+\nabla(L^\gamma-K)
\end{equation*}
where $i_j = 0$ or $1$  if the $j$th over crossing of $K$ in $L$ was a $+1$ crossing or a $-1$ crossing, respectively. 

 As $L-K$ and $L^\gamma-K$ are split, $\nabla(L-K)=\nabla(L^\gamma-K)=0$.  Moreover, every crossing we resolved was an intercomponent crossing, so $L_{0,j}$ is a link with $c-1$ components for every $j$; therefore by the inductive hypothesis, $\nabla(L_{0,j})=(-1)^c\nabla(L_{0,j}^\gamma)$.

We thus have
\begin{align*}
\nabla(L)&=(-1)^{i_1+c}z\nabla(L_{0,1}^\gamma)+(-1)^{i_2+c}z \nabla(L_{0,2}^\gamma)+...+(-1)^{i_n+c}z\nabla(L_{0,n}^\gamma)\\
&=(-1)^{c+1}\nabla(L^\gamma) 	\qedhere
\end{align*}
\end{proof}

This lemma rules out four possible symmetries for nonsplit, alternating two-component links $L$.  Recall that the Conway polynomial does not vanish for nonsplit, alternating links.  Thus such a link $L$ cannot admit a symmetry of the form $(-1,\varepsilon,\varepsilon, p)$.  Only 12 subgroups of $\Gamma_2$ fail to contain one of the four elements of this form; the other 15 groups cannot be symmetry groups for an alternating two-component link, which proves the following proposition.

\begin{proposition}
Only 12 intrinsic symmetry groups are possible for alternating, nonsplit two-component links.  These 12 groups may be viewed as a subdiagram of Figure~\ref{fig:subgroup_lattice}:  they form the order ideal generated by two groups, $\s{8}{1}$ and $\s{8}{2}$, i.e., these two groups and all of their subgroups.  
\end{proposition}

Alternating examples appear in Table~\ref{t:examples} for all but one of these 12 groups; the exception is $\s{4}{4}$, where no example, alternating or not, is known.


\section{Future directions}
We have found prime, nontrivial examples for 21 of the 27 possible intrinsic symmetry groups of two-component links..  Do examples exist for the 6 remaining subgroups?  
In considering links of 14 or fewer crossings, our exhaustive approach has not captured the full complexity of possible link structures.  It seems possible that some of the missing groups, especially the smaller ones, will appear among links of 20 or fewer crossings.  For instance, examples for three groups ($\s{4}{8}, \s{4}{10}, \s{4}{11}$) appeared first for 14-crossing links; perhaps examples of links with the three missing groups of order four will appear `soon' in the tables.  In particular, we conjecture that there is an alternating link with symmetry group $\s{4}{4}$, which would complete the set of examples for alternating, nonsplit two-component links.

More interesting is the search for nonsplit links with full symmetry.  The unlink with $\mu$ components possesses full symmetry; to date no nonsplit examples are known.  A two-component link $L$ with symmetry group $\Sigma(L) = \Gamma_2$ must have linking number zero and be nonalternating; its two knot components must be of the same knot type, which itself must possess full symmetry. We are compelled to put forth the following conjecture.

\begin{conjecture}
There does not exist a prime, nonsplit two-component link with full symmetry.
\end{conjecture}


Another project is to extend our catalog of the relative frequencies of different intrinsic symmetry groups.  To date, we have only approached hyperbolic two-component links of 14 or fewer crossings; raising the number of components is an obvious next step.  However, searching for the intrinsic symmetry groups of links with more components becomes more daunting, as the number of such groups rises considerably; see Table~\ref{t:number}.  Few examples of intrinsic symmetry groups are known for $\mu$-component links for $\mu \geq 3$; we found a handful in \cite{intrinsicsym}. However, using {\tt SnapPy}, searching for examples and developing a catalog of frequencies seems like an approachable problem.

\begin{table}[ht]
\begin{center}
\begin{tabular}{crrr}
\toprule
&&& \# subgroups  \\
$\mu$ & $\left|\Gamma_\mu\right|$ & \# subgroups & (up to conjugacy) \\
\midrule
1 &      4        &       5    & 5 \\
2  &    16       &       35 & 27                     \\
3  &     96      &       420 & 131                 \\
4  &     768    &     9417 &    994           \\
5  &     7680  &      270131 &    6382     \\
\bottomrule
\end{tabular}
\medskip
\caption[Number of subgroups of $\Gamma_\mu$; this equals the number of symmetry groups for a $\mu$-component link.] {The number of subgroups of $\Gamma_\mu$, as computed by {\tt Magma}.  Each one represents a different intrinsic symmetry group possible for a $\mu$-component link.  
}
\label{t:number}
\end{center}
\end{table}

\section*{Acknowledgements}
We are grateful to all the members of the UGA Geometry VIGRE group who contributed their time and effort over many years towards our original intrinsic symmetry project.  Jeremy Rouse assisted with our {\tt Magma} computations of subgroups of $\Gamma_\mu$.  Our work was supported by the UGA VIGRE grants DMS-07-38586 and DMS-00-89927 and by the UGA REU site grant DMS-06-49242.

\bibliographystyle{plain}
\bibliography{examplepaper,drl}
\end{document}